\documentclass[11pt]{article}
\usepackage[margin=1in]{geometry} 
\usepackage{amssymb,amsfonts,amsmath,amsthm,bbm,mathrsfs,stmaryrd}
\newtheorem{lemma}{Lemma}[section]

\newtheorem{corollary}[lemma]{Corollary}
\newtheorem{theorem}[lemma]{Theorem}
\usepackage{cleveref}

\newcommand{\Addresses}{{
		\bigskip
		\footnotesize
		
		\textsc{Department of Mathematics, University at Buffalo, Buffalo,
			NY 14260-2900, USA}\par\nopagebreak \textit{E-mail address}:
		\texttt{gagecosg@buffalo.edu}
	}}
\title{The Image of the Gassner Representation of the Pure Braid Subgroup has Pairwise Free Generators}
\author{Gage Makenzie Cosgrove}

\begin{document}
	\maketitle
	\begin{abstract}
		\noindent While much is known about the faithfulness of the Burau representation, the problem remains open for the Gassner representation for every $B_n$ with $n\geq 4$. We first find the definition of the Colored-Burau representation in \cite{aagl} and we show that this is equivalent, when restricted to the pure braid subgroup, to the Gassner representation. The methods of Abdulrahim \cite{phdthesis} and Knudson \cite{knudson} require analysis within the lower central series of a free subgroup of the pure braid group. However, Lipschutz's work \cite{lip} gives a method for analyzing the Gassner representation and the Colored-Burau structure reduces this analysis to basic linear algebra. This culminates in a proof that the image of the Gassner representation is generated by pair-wise free elements. We then discuss applications of this result to the faithfulness problem of the representation.
	\end{abstract}

\makeatletter
\renewcommand\subsubsection{\@startsection{subsubsection}{3}{\z@}%
	{-3.25ex\@plus -1ex \@minus -.2ex}%
	{-1.5ex \@plus -.2ex}
	{\normalfont\normalsize\bfseries}}

\subsubsection*{Keywords:}
Braid groups, pure braids, Gassner representation, Ping-Pong Lemma

\makeatother 

\subsection*{Acknowledgments}
Many thanks to my advisor, Alexandru Chirvasitu, for his insights. Special thanks to Arya Vadnere for his suggestion to consider the Ping-Pong Lemma, and also to Johanna Mangahas for sharing her expertise on the subject.

\section{Introduction}
The faithfulness of the Burau representation has been well-studied. It is known that the representation is faithful for $n=3$ and unfaithful for $n\geq 5$ \cite{big-bur}. The case $n=4$ is still open, but several attempts have been made to answer the question. Beridze and Traczyk showed that the faithfulness depends on the restriction to a free subgroup generated by two elements and that the cubes of their images generate a free group \cite{ber-tra}. Their proof was topological in nature, however and, more recently, Beridze, Bigelow, and Traczyk gave an alternate proof using the Ping-Pong Lemma that the same holds modulo $p$ \cite{ber-big-tra}. This has subsequently been improved by Abdulrahim and Jinan in \cite{abdul-jin} to show the squares generate a free group.
\\
\\
When it comes to the Gassner representation, it is not known for any $n\geq 4$ whether the representation is faithful. In fact, the faithfulness of the Gassner representation is dependent only upon its faithfulness when restricted to free subgroups of the pure braid group. For example, Kundson located the kernel within the lower central series of a certain free subgroup of the pure braid group \cite{knudson}. The Gassner representation is originally constructed as a Magnus representation using Fox derivatives and the computations can be seen in \cite{phdthesis}. These computations are strenuous and the resulting matrices require additional rewriting to make use of them.
\\
\\
In this paper, we prefer to use the Colored-Burau representation introduced in \cite{aagl} over Fox derivatives. Rather than obtaining a representation via homology computations as in the ordinary Burau representation, the Colored-Burau representation is simply an imposition of algebraic structure. In \Cref{sec:cb}, we introduce the Colored-Burau representation, compute the images of the pure braid generators found in \cite{Birm}, and show these align with the Gassner matrices. In \Cref{sec:lipschutz}, we use the preceding computations to provide a short linear-algebraic proof that the faithfulness of the Gassner representation depends only upon its restriction to a free subgroup. Finally, in \Cref{sec:main}, we show, following the idea in \cite{lip}, that the Gassner matrices for the generators of the free subgroup are pairwise free.

\section{The Colored-Burau Representation}\label{sec:cb}
Here we introduce the Colored-Burau representation from \cite{aagl}. This representation is a generalization of the famous Burau representation in that it is defined for the entire braid group. However, we use this represenation to perform explicit computations on the pure braid subgroup. These computaions will show that the restriction of the Colored-Burau representation to the pure braid subgroup is precisely the Gassner representation. In \Cref{sec:lipschutz} and \Cref{sec:main} we will leverage this to write short linear-algebraic proofs about the Gassner representation.

\subsection{A Convention}\label{subsec:convention}
Fix an integer $n\geq 2$. For $1\le i<j\le n$ and a scalar or variable $t$ we denote by 
$t_{i\to j}$ the column vector (of implicit size $n$) with $t$ in positions $i$ up to 
and including $j-1$ and $0$ elsewhere. On the other hand, given a length-$n$ column 
vector ${\bf v}$, we write $c_s({\bf v})$ for the $n\times n$ matrix whose $s$-column is 
${\bf v}$ (and all of whose other entries vanish). So, we could write

\begin{equation*}
c_2([1,2,3]^T)=
\begin{bmatrix}
0&1&0\\
0&2&0\\
0&3&0
\end{bmatrix}.
\end{equation*}
This example also illustrate the convention that we number rows and columns starting 
at $1$ rather than $0$.

\subsection{The Colored-Burau Group}
Fix an integer $n\geq 2$. Denote by $B_n$ the Braid group on $n$ strands.
Let $T=(t_1,\hdots t_n)$ be a tuple of formal variables. Denote by
$L_n:=\mathbb{Z}[t_1^{\pm 1},\hdots, t_n^{\pm 1}]$, the integral Laurent polynomial ring in $n$ 
variables. 
\\
\\
The Colored-Burau group is the outer semi-direct product 
$$G=GL_n(L_n)\rtimes S_n$$
of the symmetric group on $n$ symbols acting on $L_n$. The action is by permutation of 
the variables. In particular, there is a homomorphism $\pi:B_n\to S_n$ given by 
$s_i:=\pi(\sigma_i)=(i,i+1)$, so that a braid $\beta\in B_n$ may act on an element of $M\in GL_n(L_n)$ as
$${}^{\beta}M:={}^{\pi(\beta)}M$$
Accordingly, the multiplication in $G$ is
$$(M,s)\star (M',s'):=(M{}^sM',ss')$$
The Colored-Burau representation, therefore, is the map
$$CB:B_n\to G$$
given by 
$$\sigma_i\mapsto (M_i,s_i)$$
where
$$M_1=\begin{bmatrix}
-t_1 & 1 & &\\
& 1 & & \\
& & \ddots &\\
& & & 1
\end{bmatrix} 
\qquad \text{and, when $i\neq 1$,} \qquad 
M_i=
\begin{bmatrix}
1 & & &\\
& \ddots & \\
& & t_i& -t_i & 1 & &\\
& & & & & \ddots &\\
& & & & & & 1
\end{bmatrix}$$
which is the identity except for the $i^{\text{th}}$ row where it has in succession 
$t_i, -t_i, 1$ with $-t_i$ on the diagonal. Using the convention from 
\Cref{subsec:convention}, we can write:
$$\sigma_i \mapsto I_n+c_{i-1}((t_i)_{i-1\to i})+c_i((1-t_i)_{i\to i+1})+c_{i+1}(1_{i+1\to i+2})$$
(when $i=1$, the $c_{i-1}$ term vanishes because we index columns from $1$, not $0$). For convenience, given a braid $\beta\in B_n$, we will refer to the matrix component of $CB(\beta)$ as $cb(\beta)$.
\\
\\
One can check that $CB$ is a homomorphism, in particular, it satisfies the braid relations, so it is well-defined. Finally, note that applying the quotient $t_i\sim t$ that identifies all $t_i$ with a single indeterminant will recover the Burau representation of $B_n$.

\subsection{Matrices for Pure Braid Generators}\label{subsec:cb-xpl}
Among other sources, Birman shows in \cite{Birm} the Pure Braid subgroup is generated by
	$$A_{i,j}=\sigma_i\sigma_{i+1}\cdots \sigma_{j-2}\sigma_{j-1}^2\sigma_{j-1}^{-1}\cdots \sigma_i^{-1}$$
together with relations that we do not list here. Moreover, Birman shows there is a split exact 
sequence
	$$1\to F_{n-1}\to P_n\twoheadrightarrow P_{n-1}\to 1$$
such that $F_{n-1}\cong \langle A_{1,j}\rangle_{j=2}^n$. As such, we describe explicitly the matrices $cb(A_{i,j})$ and their inverses. The reader can check that the formula derived here agrees precisely with Abdulrahim's result (see \cite[Proposition 3.3.3]{phdthesis}), however, our derivation is via computations in the Colored-Burau group.

\begin{lemma}\label{lem:cb-xpl}
	Let $n\ge 2$ and
	\begin{equation*}
	A_{i,j},\ 1\le i<j\le n
	\end{equation*}
	be a pure braid generator for some $i,j$. The matrix
	\begin{equation*}
	cb(A_{i,j})\in GL_n(L_n)  
	\end{equation*}
	is equal to
	\begin{eqnarray*}
		I_n &+\quad c_{i-1}((-t_it_j+t_i)_{i\to j}) &+\quad c_i((t_j-1)_{i\to j})\\
		&+\quad c_{j-1}((t_it_j-t_j)_{i\to j}) &+\quad c_{j}((-t_i+1)_{i\to j}),
	\end{eqnarray*}
	with the following caveat: When $i=1$ the term $c_{i-1}$ is absent.
\end{lemma}

\begin{proof}
	When $n=2$, we need only check $A_{1,2}$, which we obtain as the result of a single
	group multiplication:
	\begin{align*}
		CB(A_{1,2})&=CB(\sigma_1^2)\\
		&=
		\left(
		\begin{bmatrix}
			-t_1 & 1\\
			0 & 1
		\end{bmatrix},(1,2)
		\right)
		\star
		\left(
		\begin{bmatrix}
			-t_1 & 1\\
			0 & 1
		\end{bmatrix},(1,2)
		\right)\\
		&=
		\left(
		\begin{bmatrix}
			-t_1 & 1\\
			0 & 1
		\end{bmatrix}
		\cdot 
		\begin{bmatrix}
			-t_2 & 1\\
			0 & 1
		\end{bmatrix},(1,2)\cdot (1,2)
		\right)\\
		&=
		\left(
		\begin{bmatrix}
			t_1t_2 & 1-t_1\\
			0 & 1
		\end{bmatrix},e
		\right)
	\end{align*}
	and the matrix component is equal to
	$$
	\underbrace{\begin{bmatrix}
			1 & 0 \\
			0 & 1
	\end{bmatrix}}_{I_2}
	+
	\underbrace{\begin{bmatrix}
			t_2-1 & 0 \\
			0 & 0
	\end{bmatrix}}_{c_1((t_2-1)_{1\to 2})}
	+
	\underbrace{\begin{bmatrix}
			t_1t_2-t_2 & 0\\
			0 & 0
	\end{bmatrix}}_{c_1((t_1t_2-t_2)_{1\to 2})}
	+
	\underbrace{\begin{bmatrix}
			0 & 1-t_1\\
			0 & 0
	\end{bmatrix}}_{c_2((1-t_1)_{1\to 2})}
	=
	\begin{bmatrix}
		t_1t_2 & 1-t_1\\
		0 & 1
	\end{bmatrix}
	$$
	as claimed.
	\\
	\\
	Now suppose the result holds for some $B_k$ with $k>2$ and for all 
	$1\leq i<j \leq k$. Set $n=k+1$. By our inductive hypothesis, it suffices to prove 
	the formula for $j=n$. Thus, we further induct on $\ell=n-i$. When $\ell=1$, we 
	obtain the result immediately by one group multiplication:
	\begin{align*}
		CB(A_{n-1,n})&=CB(\sigma_{n-1}^2)=CB(\sigma_{n-1})\star CB(\sigma_{n-1})\\
		&=
		\left(
		\begin{bmatrix}
			I_{n-3}&\\
			& 1 & 0 & 0\\
			& t_{n-1} & -t_{n-1} & 1\\
			& 0& 0& 1
		\end{bmatrix},(n-1,n)
		\right)
		\star
		\left(
		\begin{bmatrix}
			I_{n-3}&\\
			& 1 & 0 &0\\
			& t_{n-1} & -t_{n-1}&1\\
			& 0& 0& 1
		\end{bmatrix},(n-1,n)
		\right)\\
		&=
		\left(
		\begin{bmatrix}
			I_{n-3}&\\
			& 1 & 0&0\\
			& t_{n-1} & -t_{n-1}&1\\
			& 0& 0& 1
		\end{bmatrix}
		\cdot
		\begin{bmatrix}
			I_{n-3}&\\
			& 1 & 0&0\\
			& t_n & -t_n&1\\
			& 0& 0& 1
		\end{bmatrix},(n-1,n)\cdot(n-1,n)
		\right)\\
		&=
		\left(
		\begin{bmatrix}
			I_{n-3}&\\
			& 1 & 0&0\\
			& t_{n-1}-t_{n-1}t_n & t_{n-1}t_n&1-t_{n-1}\\
			& 0& 0& 1
		\end{bmatrix}, e
		\right)
	\end{align*}
	and the matrix component is equal
	\begin{align*}
		&\underbrace{\begin{bmatrix}
				I_{n-3} & 0 & 0 & 0\\
				0 & 1 & 0 & 0 \\
				0 & 0 & 1 & 0 \\
				0 & 0 & 0 & 1
		\end{bmatrix}}_{I_n}
		+\underbrace{\begin{bmatrix}
				0 & 0 & 0 & 0\\
				0 & 0 & 0 & 0 \\
				0 & -t_{n-1}t_n+t_{n-1} &0 & 0 \\
				0 & 0 & 0 & 0
		\end{bmatrix}}_{c_{n-2}((-t_{n-1}t_n+t_{n-1})_{n-1\to n})}
		+\underbrace{\begin{bmatrix}
				0 & 0 & 0 & 0\\
				0 & 0 & 0 & 0 \\
				0 & 0 & t_n-1 & 0 \\
				0 & 0 & 0 & 0
		\end{bmatrix}}_{c_{n-1}((t_n-1)_{n-1\to n})}\\
		&+\underbrace{\begin{bmatrix}
				0 & 0 & 0 & 0\\
				0 & 0 & 0 & 0 \\
				0 & 0 & t_nt_{n-1}-t_n & 0 \\
				0 & 0 & 0 & 0
		\end{bmatrix}}_{c_{n-1}((t_nt_{n-1}-t_n)_{n-1\to n})}
		+\underbrace{\begin{bmatrix}
				0 & 0 & 0 & 0\\
				0 & 0 & 0 & 0 \\
				0 & 0 & 0 & -t_{n-1}+1 \\
				0 & 0 & 0 & 0
		\end{bmatrix}}_{c_n((-t_{n-1}+1)_{n-1\to n})}\\
		&=
		\begin{bmatrix}
			I_{n-3}&\\
			& 1 & 0&0\\
			& t_{n-1}-t_{n-1}t_n & t_{n-1}t_n&1-t_{n-1}\\
			& 0& 0& 1
		\end{bmatrix}
	\end{align*}
	as claimed.
	\\
	\\
	Now suppose our formula holds for some $2\leq \ell \leq n-2$ 
	(we treat the case $i=1$ separately). Since 
	$$A_{i-1,n}=\sigma_{i-1}A_{i,n}\sigma_{i-1}^{-1},$$ 
	we perform two group multiplications.
	Firstly, since $A_{i,n}$ is a pure braid, there is no permuting of the variables in 
	$cb(\sigma_{i-1}^{-1})$. Thus, we multiply
	\begin{align*}
		cb(A_{i,n})cb(\sigma_{i-1}^{-1})&=
		\begin{bmatrix}
			I_{i-2}&\\
			&1 & 0 & 0 & 0&0\\
			&t_i(1-t_n)&t_n & 0 & t_it_n-t_n&1-t_i\\
			&t_i(1-t_n)&t_n-1 & I_{\ell-2} & t_it_n-t_n&1-t_i\\
			&\vdots &\vdots & & \vdots&\vdots \\
			&t_i(1-t_n)&t_n-1 & 0 & 1+t_it_n-t_n &1-t_i\\
			& 0& 0& 0& 0& 1	
		\end{bmatrix}\cdot
		\begin{bmatrix}
			I_{i-3} &  &  &  \\
			& 1 & 0 & 0\\
			& 1& \frac{-1}{t_i} & \frac{1}{t_i}\\
			& 0 & 0 & 1 \\
			& & & & I_{n-i}
		\end{bmatrix}
	\end{align*}
	The product is non-unital in precisely the columns $i-1,i,i+1,$ and $n-1$ with the entries being
	$$\begin{bmatrix}
		I_{i-3}& \\
		&1 & 0 & 0  & 0 & 0 & 0\\
		&1& \frac{-1}{t_i} & \frac{1}{t_i} & 0 & 0 &0\\\\
		&t_i(1-t_n)&t_n-1 & 1 & 0 &t_it_n-t_n&1-t_i\\
		&\vdots & \vdots & I_{\ell-2}&\vdots & \vdots &\vdots \\
		&t_i(1-t_n)&t_n-1 & 0 & 0 & 1+t_it_n-t_n &1-t_i\\
		&0& 0& 0& 0& 0& 1
	\end{bmatrix}.$$
	The second group multiplication has the effect of permuting every occurrence of 
	$t_i$ in the above to $t_{i-1}$ and subsequently multiplying by $cb(\sigma_{i-1})$, 
	hence we obtain
	$${\small\begin{bmatrix}
			I_{i-3}\\
			&1& 0 & 0 \\
			&t_i & -t_i & 1\\
			&0 & 0 & 1\\
			& & & & I_{n-i}
	\end{bmatrix}}\cdot
	\begin{bmatrix}
		I_{i-3}& \\
		&1 & 0 & 0  & 0 & 0 & 0\\
		&1& \frac{-1}{t_{i-1}} & \frac{1}{t_{i-1}} & 0 & 0 &0\\\\
		&t_{i-1}(1-t_n)&t_n-1 & 1 & 0 &t_{i-1}t_n-t_n&1-t_{i-1}\\
		&\vdots & \vdots & I_{\ell-2}& \vdots & \vdots &\vdots \\
		&t_{i-1}(1-t_n)&t_n-1 & 0 & 0 & 1+t_{i-1}t_n-t_n &1-t_{i-1}\\
		&0& 0& 0& 0& 0& 1
	\end{bmatrix},$$
	which is non-unital in precisely columns $i-1,i,n-1,$ and $n$ with the entries being
	$$\begin{bmatrix}
		I_{i-2}&\\
		&1 & 0 & 0 & 0&0\\
		&t_{i-1}(1-t_n)&t_n & 0 &t_{i-1}t_n-t_n&1-t_{i-1}\\
		&t_{i-1}(1-t_n)&t_n-1 & I_{\ell-2} & t_{i-1}t_n-t_n&1-t_{i-1}\\
		&\vdots &\vdots & & \vdots &\vdots \\
		&t_{i-1}(1-t_n)&t_n-1 & 0 & 1+t_{i-1}t_n-t_n & 1-t_{i-1}\\
		&0 & 0 & 0 & 0 & 1
	\end{bmatrix}$$
	and this agrees with the proposed formula, as the reader can verify.
	\\
	\\
	Finally, we handle the case $i=1$. By induction, we have 
	$$cb(A_{2,n})=
	\begin{bmatrix}
		1 & 0 & 0 & 0&0\\
		t_2(1-t_n)&t_n & 0 &t_2t_n-t_n&1-t_2\\
		t_2(1-t_n)&t_n-1 & I_{n-4} & t_2t_n-t_n&1-t_2\\
		\vdots &\vdots & & \vdots&\vdots \\
		t_2(1-t_n)&t_n-1 & 0 & 1+t_2t_n-t_n&1-t_2\\
		0 & 0 & 0 & 0& 1
	\end{bmatrix}
	$$
	and, since $A_{1,n}=\sigma_1A_{2,n}\sigma_1^{-1}$, we obtain 
	$$cb(A_{1,n})=
	\begin{bmatrix}
		t_n & 0 &t_1t_n-t_n&1-t_1\\
		t_n-1 & I_{n-3} & t_1t_n-t_n&1-t_1\\
		\vdots & & \vdots & \vdots\\
		t_n-1 & 0 & 1+t_1t_n-t_n&1-t_1\\
		0 & 0 & 0 & 1
	\end{bmatrix},
	$$
	as the reader can verify. Finally, we note that this agrees
	with the proposed formula.
\end{proof}

\begin{corollary}\label{cor:inv}
	Using the same notation as in \Cref{thm:cb-xpl}, the matrix
	\begin{equation*}
		cb(A_{ij})^{-1}\in GL_n(L_n)  
	\end{equation*}
	is equal to
	$$
	I_n +c_{i-1}\left(\left(\frac{t_j-1}{t_j}\right)_{i\to j}\right)
	+ c_i\left(\left(\frac{1-t_j}{t_it_j}\right)_{i\to j}\right)
	+ c_{j-1}\left(\left(\frac{1-t_i}{t_i}\right)_{i\to j}\right) 
	+ c_{j}\left(\left(\frac{t_i-1}{t_it_j}\right)_{i\to j}\right),
	$$
	with the following caveat: When $i=1$ the first term $c_{i-1}$ is absent.
\end{corollary}

\begin{corollary}\label{cor:det}
	Using the same notation as in \Cref{lem:cb-xpl}, taking determinants gives
		$$\det(cb(A_{i,j}))=t_it_j.$$
\end{corollary}

\section{The Faithfulness of the Gassner Representation}\label{sec:lipschutz}
The main result of this section is that the faithfulness of the Gassner representation is determined completely by its faithfulness on the free subgroup $F_{n-1}\cong \langle A_{1,j}\rangle_{j=2}^n$. The following result due to Long has provided for such reductions in many other contexts. In fact, \cite{cp} uses a similar argument on the level of Lie Algebras.

\begin{lemma}[Thereom 2.2 in \cite{lng}]\label{lem:faith-redux}
	Let $A$ be a braid group and $\rho:A\to M$ be a representation. Furthermore, suppose $\rho$
	is faithful on $Z(A)$. Then $\rho$ is faithful on $A$ if and only if it is faithful on 
	$B\lhd A$ for some non-trivial, non-central, normal subgroup $B$.
\end{lemma}

\noindent In \cite{knudson}, Knudson uses the restriction to a free subgroup to show the Gassner representation is faithful if
$$\Gamma^iF_{n-1}/\Gamma^{i+1}F_{n-1}\to\Gamma^ig_n(F_{n-1})/\Gamma^{i+1}g_n(F_{n-1})$$ 
is injective for all $i$ (here $g_n$ is the restriction of Gassner to $F_{n-1}$). However, the injectivity fails for $i=5$. Nevertheless, he shows the kernel is contained in $[\Gamma^3F_{n-1},\Gamma^2F_{n-1}]$. 
\\
\\
While the result given in this section is not novel, the proof requires only the computation of a determinant, which is simple for Colored-Burau matrices. This allows us to ultimately circumvent the computations in the lower central series of $F_{n-1}$ by performing linear algebraic operations on the matrices for the pure braid generators directly. We require only the following result about centers of braid groups. 

\begin{lemma}\label{lem:zpn=zbn}
	For $n\geq 3$, $Z(B_n)=Z(P_n)$.
\end{lemma}

\begin{proof}
	Certainly any braid $\beta\in Z(B_n)$ must commute with all pure braids, hence 
	$Z(B_n)\subseteq Z(P_n)$. The reverse containment is a consequence of in \cite[Corollary 1.8.4]{Birm},
	where it is shown that $Z(B_n)$ is generated by the pure braid 
	$\sigma =A_{1,2}(A_{1,3}A_{2,3})\cdots(A_{1,n}\cdots A_{n-1,n})$.
\end{proof}

\begin{theorem}\label{thm:pn-faithful}
	The Gassner representation on $P_n$ is faithful if and only if it is so on
	$F_{n-1}\leq P_n$.
\end{theorem}

\begin{proof}
	Certainly the Gassner representation is faithful on $F_{n-1}$ if it is so on $P_n$. Thus, it suffices to show the converse, which we do using \Cref{lem:faith-redux}. Since $F_{n-1}$ is normal, being that it is the kernel of a homomorphism in the split exact sequence
		$$1\to F_{n-1}\to P_n\twoheadrightarrow P_{n-1}\to 1,$$
	and because $F_{n-1}$ is non-trivial and non-central in $P_n$, we need only prove that the Gassner representation is faithful on $Z(P_n)$.
	\\
	\\
	Since any braid in $Z(P_n)$ must have an image that is central in $GL_n(L_n)$, the image must be scalar. Using \Cref{lem:zpn=zbn} and Corollary 2.3, we obtain
		$$\det(cb(\sigma))=\prod_{j=2}^nt_j^{j-1}\prod_{i=1}^{j}t_i=\prod_{i=1}^nt_i^{n-1},$$
	so that $cb(\sigma)$ has infinite order. Since $Z(B_n)=Z(P_n)$ and $\sigma$ generates the former, $CB$ must be faithful on the latter. By \Cref{lem:cb-xpl}, $CB$ is the same as the Gassner representation and this completes the proof.
\end{proof}

\section{The Main Result}\label{sec:main}
Now we show that the images of the Pure Braid generators $cb(A_{1j})$ are pairwise free. The proof uses the Ping-Pong Lemma, which we record here for reference. A proof of this can be found, for example, in \cite{geo}.

\begin{lemma}[Ping-Pong]\label{lem:pp}
	Let $G$ be a group acting on a set $X$ and let $H_1,H_2\leq G$ be subgroups. Let $|H_i|>2$ for some $i$. If $X_1,X_2\subseteq X$ are disjoint and nonempty and $H_1(X_2)\subseteq X_1$ and $H_2(X_1)\subseteq X_2$, then $G\cong H_1\ast H_2$.
\end{lemma}

\noindent With the Ping-Pong Lemma established, we shall now use it to prove that the matrices $\{cb(A_{1,j})\}_{j=2}^n$ are pairwise free. Firstly, we observe that, for $j\neq j'$, any relation held by $cb(A_{1,j}),cb(A_{1,j'})$ must hold for all tuples $T=(t_1,\hdots,t_n)$. Thus, if we can find a single tuple $T=(\tau_1,\hdots,\tau_n)$ for which the resulting matrices are free, then $cb(A_{1,j})$ and $cb(A_{1,j'})$ must be free as well. Secondly, the freeness is independent of basis: if
	$$\prod_{k=1}^N cb(A_{1,i_k})^{\alpha_k}=I_n$$
where $i_k\in \{j,j'\}$ for all $k$, then the change of basis matrix $P$ gives
	$$\prod_{k=1}^N Pcb(A_{1,i})^{\alpha_k}P^{-1}=P\left(\prod_{k=1}^N cb(A_{1,i_k})^{\alpha_k}\right)P^{-1}=PP^{-1}=I_n.$$
Thus, we shall evaluate $cb(A_{1,j}),cb(A_{1,j'})$ and let the resulting matrices in $GL_n(\mathbb{Z})$ act on $\mathbb{R}^n$ after a change of basis that reduces the action on a fixed subspace (depending on $j,j'$) to that of the following free matrices.

\begin{lemma}\label{lem:sl2-Z}
	Let $G=SL_2(\mathbb{Z})$ and 
	$$A=\begin{bmatrix}
		1&2\\
		0&1
	\end{bmatrix} \qquad 
	B=\begin{bmatrix}
		1&0\\
		2&1
	\end{bmatrix}$$
	Then $\langle A,B\rangle \cong F_2$ is free.
\end{lemma}

\begin{proof}
	We use the Ping-Pong Lemma. Let 
	$$H_1=\{A^n: n\in \mathbb{Z}\}=\left\lbrace \begin{bmatrix}1&2n\\0&1\end{bmatrix}: n\in\mathbb{Z}\right\rbrace \qquad \text{and} \qquad H_2=\{B^n: n\in \mathbb{Z}\}=\left\lbrace \begin{bmatrix}1&0\\2n&1\end{bmatrix}: n\in\mathbb{Z}\right\rbrace .$$
	
	Consider the standard action of $G$ on $X=\mathbb{R}^2$ by linear transformations. Set 
	$$X_1=\left\lbrace \begin{bmatrix} x\\y \end{bmatrix} \in \mathbb{R}^2 : |x|>|y|\right\rbrace 
	\qquad \text{and}\qquad X_2=\left\lbrace \begin{bmatrix} x\\y \end{bmatrix}\in\mathbb{R}^2: |x|<|y|\right\rbrace$$
	
	\noindent Now, if ${\bf v}=\begin{bmatrix} x\\y \end{bmatrix}\in X_2$, then $A^n{\bf v}=\begin{bmatrix} x+2ny\\y\end{bmatrix}\in X_1$ since
		$$|x+2ny|\geq |2ny|-|x|\geq 2|y|-|x|>2|y|-|y|=|y|.$$
	Likewise, if ${\bf v}=\begin{bmatrix} x\\y \end{bmatrix}\in X_1$, then $B^n{\bf v}=\begin{bmatrix} x\\2nx+y\end{bmatrix}\in X_2$ since
		$$|2nx+y|\geq |2nx|-|y|\geq 2|x|-|y|>2|x|-|x|=|x|.$$
\end{proof}

\noindent These are the same matrices that appear in blocks in the representation given in \cite{lip} after evaluating at $t_i=-1$ for all $i$. In order to describe the change of basis explicitly, we evaluate the matrices $cb(A_{1,j})$ and $cb(A_{1,j'})$, and compute their eigenspaces. The new basis is then given in terms of the eigenvectors.

\begin{lemma}\label{lem:eigen}
	Let $n\geq 4$. Denote by $\{e_k\}_{k=1}^n$ the standard basis for $\mathbb{R}^n$. Define $\tau$ to be the evaluation homomorphism given by taking $t_i=-1$ for all $1\leq i \leq n$. Then the only eigenvalue for $\tau(cb(A_{1,j}))$ is $\lambda=1$ and the eigenvectors are
	\begin{itemize}
		\item for $j=2$: $e_k$ for $k\neq 2$.
		\item for $j\geq 3$: $e_k$ for $k\notin\{1,j-1,j\}$, $v_j=e_1+e_{j-1}$, and $w_j=e_1+e_j$.
	\end{itemize}
				
\end{lemma}

\begin{proof}
	This is an immediate consequence of \Cref{lem:cb-xpl}.
\end{proof}

\begin{theorem}\label{thm:main-thm}
	Let $n\geq 4$. For all $2\leq j<j'\leq n$, the group 
	$\langle cb(A_{1,j}),cb(A_{1,j'})\rangle \cong F_2$ is free.
\end{theorem}

\begin{proof}
For ease of notation, let $M_j:=\tau(cb(A_{1,j}))$,  where $\tau$ is the evaluation homomorphism $t_i=-1$ for all $1\leq i\leq n$ and let $\{e_k\}_{k=1}^n$ be the standard basis for $\mathbb{R}^n$. We begin with the case $j=2$. If $j'=3$, consider $A_{1,2}$ and $A_{1,3}$ as elements of $B_3$. Here, the Burau representation is known to be faithful \cite{big-bur}, so the Colored-Burau representation must be likewise. Consequently, evaluating gives that $M_2$ and $M_3$ are free. Now assume that $j'>3$. By \Cref{lem:eigen}, we can make the following change of basis by replacing $e_{j'-1}$ with $w_{j'-1}=e_1+e_{j'-1}$. We thereby obtain the following blocks in the upper left $2\times 2$ minors

$$X=\begin{bmatrix}
1 & 2\\
0 & 1
\end{bmatrix} \qquad 
Y=\begin{bmatrix}
1 & 0\\
-2 & 1
\end{bmatrix},$$
for $M_2$ and $M_{j'}$, respectively. The matrices $X$ and $Y^{-1}$ are those from \Cref{lem:sl2-Z}. Thus, $M_2$ and $M_{j'}$ cannot have a relation as this would induce a relation on $X$ and $Y^{-1}$.
\\
\\
If $j>2$, we may simultaneously replace $e_{j-1}$ with 
$v_j$ and $e_j$ with $\sum\limits_{i=1}^{j-1}e_i$ to obtain 
$$
\begin{bmatrix}
	* & 0 & *\\
	* & X & *\\
	* & 0 & *\\
\end{bmatrix}
\qquad \text{and} \qquad
\begin{bmatrix}
	* & 0 & *\\
	* & Y & *\\
	* & 0 & *\\
\end{bmatrix}
$$
for $M_j$ and $M_{j'}$, respectively, in the $2\times 2$ minor with the upper left entry at position $(j-1,j-1)$. The same argument from the preceding paragraph shows that $M_j$ and $M_{j'}$ must not generate a relation, hence are free.
\end{proof}

\section{Future Research}
For $n\geq 4$, the faithfulness of the Gassner representation is unknown. In this section we shall discuss the application of the main result in this paper to the special case $n=4$. In this case, we have 
$F_3\cong \langle A_{1,2},A_{1,3},A_{1,4}\rangle$. Theorem 4.4 shows that the images of any two of these generate the free group $F_2$. However, it is unclear whether the three images generate $F_3$. In particular, we have
\pagebreak
	$$cb(A_{1,2})=
	\begin{bmatrix}
	t_1t_2 & 1-t_1 & 0 & 0 \\
	0 & 1 & 0 & 0\\
	0 & 0 & 1 & 0\\
	0 & 0 & 0 & 1
	\end{bmatrix}$$
	
	$$cb(A_{1,3})=
	\begin{bmatrix}
	t_3 & t_3(t_1-1) & (1-t_1) & 0 \\
	t_3-1 & 1+t_3(t_1-1) & (1-t_1) & 0\\
	0 & 0 & 1 & 0\\
	0 & 0 & 0 & 1
	\end{bmatrix}$$

	$$cb(A_{1,4})=
	\begin{bmatrix}
	t_4 & 0 & t_4(t_1-1) & (1-t_1) \\
	t_4-1 & 1 & t_4(t_1-1) & (1-t_1)\\
	t_4-1 & 0 & 1+t_4(t_1-1) & (1-t_1)\\
	0 & 0 & 0 & 1
	\end{bmatrix}$$ 
	
	\noindent 
	For example, when $j=2$ and $j'=4$ taking $t_i=-1$ and changing to the basis $\{e_1,e_2,e_1+e_3,e_4\}$, we obtain the matrices 
	$$M_2:=\tau(cb(A_{1,2}))=
	\begin{bmatrix}
	1&2&0&0\\
	0&1&0&0\\
	0&0&1&0\\
	0&0&0&1
	\end{bmatrix}
	\qquad
	M_4:=\tau(cb(A_{1,4}))=
	\begin{bmatrix}
	1&0&1&0\\
	-2&1&0&2\\
	-2&0&1&2\\
	0&0&0&1
	\end{bmatrix},
	$$
	which are free by \Cref{lem:sl2-Z}. Likewise, changing to the basis $\{e_1,e_1+e_2,e_1+e_2+e_3,e_4\}$ and evaluating at $t_i=-1$, we can see that 
	$$M_3:=\tau(cb(A_{1,3}))=
	\begin{bmatrix}
	1&0&0&0\\
	-2&1&2&0\\
	0&0&1&0\\
	0&0&0&1
	\end{bmatrix}
	\qquad
	M_4:=\tau(cb(A_{1,4}))=
	\begin{bmatrix}
	1&0&0&0\\
	0&1&0&0\\
	-2&-2&1&2\\
	0&0&0&1
	\end{bmatrix}
	$$ 
	are free by \Cref{lem:sl2-Z}.
	\\
	\\
	However, the fact that each of these pairs is free does not imply that the set of three is free. In fact, using the basis
	basis $\{b_1,b_2,b_3,b_4\}$, where $b_k=\sum_{i=1}^k e_i$, we obtain
	$$M_2=
	\begin{bmatrix}
	1 & 2 & 2 & 2 \\
	0 & 1 & 0 & 0\\
	0 & 0 & 1 & 0\\
	0 & 0 & 0 & 1
	\end{bmatrix}
	\qquad 
	M_3=
	\begin{bmatrix}
	1& 0 & 0 & 0 \\
	-2 & 1 & 2 & 2\\
	0 & 0 & 1 & 0\\
	0 & 0 & 0 & 1
	\end{bmatrix}
	\qquad
	M_4=
	\begin{bmatrix}
	1 & 0 & 0 & 0 \\
	0 & 1 & 0 & 0\\
	-2 & -2 & 1 & 2\\
	0 & 0 & 0 & 1
	\end{bmatrix}$$
	and it is still clear these are pairwise-free, but the reader can verify that the following relation holds:
		$$M_2M_3M_2M_4M_2^{-1}M_3^{-1}M_2^{-1}M_3^{-1}M_4^{-1}M_3 = I_4$$
	The same word in the Burau matrices is not the identity, but this relation appears in the evaluated matrices as a result of the fact that $t=-1$ is a root of certain polynomials appearing in the matrices.
	\\
	\\
	It is not evident to the author that the Ping-Pong Lemma is a sufficient tool to completely resolve the question of faithfulness even in this case. Similar approaches have been taken in \cite{ber-big-tra} and \cite{abdul-jin} regarding the faithfulness of the Burau representation for $n=4$. While the generators used in these papers are different, the authors have managed to show that the cubes and squares, respectively, of the images are free, rather than the images themselves.

\subsection*{Data Availability}
Data sharing not applicable to this article as no datasets were generated or analysed during the current study.

\subsection*{Compliance with Ethical Standards}
The author has no financial or proprietary interests in any material discussed in this article.

\bibliography{bib}{}
\bibliographystyle{plain}
\addcontentsline{toc}{section}{References}
\Addresses

\end{document}